\ifdefined\pdfoutput\pdfoutput=1\fi % PDF output for arXiv.
\documentclass[11pt]{amsart}

\usepackage{amssymb}
\usepackage{mathrsfs}
\usepackage{amsfonts}
\usepackage{latexsym,amsmath,amsthm,amssymb,amsfonts}
\usepackage[usenames]{color}
\usepackage{xspace,colortbl}
\usepackage{graphicx}
\usepackage{tipa}

\newcommand{\be}{\begin{equation}}
\newcommand{\ee}{\end{equation}}
\newcommand{\beq}{\begin{eqnarray}}
\newcommand{\eeq}{\end{eqnarray}}

\newcommand{\uS}{\mathbb{S}^{n}}

\def\R{{\mathfrak R}}

\newtheorem{prop}{Proposition}[section]

\newtheorem{remark}[prop]{Remark}

\def\begeq{\begin{equation}}
\def\endeq{\end{equation}}

\def\R{\mathbb R}
\def\S{\mathbb S}

\def\tr{{\rm tr}}

\def\odot{\setbox0=\hbox{$\bigcirc$}\relax \mathbin {\hbox
to0pt{\raise.5pt\hbox to\wd0{\hfil $\wedge$\hfil}\hss}\box0 }}

\numberwithin{equation} {section}

\numberwithin{equation}{section}
\newtheorem{theorem}{\bf Theorem}[section]
\newtheorem{proposition}[theorem]{\bf Proposition}

\newtheorem{lemma}[theorem]{\bf Lemma}

\newtheorem{corollary}[theorem]{\bf Corollary}

\allowdisplaybreaks

\usepackage{soul}
\definecolor{revisionblue}{rgb}{0.85,0.94,1.00}
\newif\ifrevisionmarks
\revisionmarkstrue

\soulregister\cite7
\soulregister\ref7
\usepackage[normalem]{ulem}
\usepackage{iftex}
\ifPDFTeX
  \usepackage[T1]{fontenc}
  \usepackage[nomath]{lmodern}
\fi
\ifLuaTeX
  \usepackage{luacolor}
  \usepackage{lua-ul}
\else

\fi
\newif\ifauditmarks
\auditmarkstrue

\makeatletter
\newcommand{\auditmathadd}[1]{\ifauditmarks\mathpalette\audit@ma{#1}\else#1\fi}
\newcommand{\audit@ma}[2]{\begingroup\setlength{\fboxsep}{0.65pt}%
  \colorbox{revisionblue}{$\m@th#1#2$}\endgroup}
\newcommand{\auditmathdel}[1]{\ifauditmarks\mathpalette\audit@md{#1}\fi}
\newcommand{\audit@md}[2]{\text{\sout{\ensuremath{#1#2}}}}
\makeatother

\soulregister\eqref7
\soulregister\textnormal1
\newcommand{\auditsep}{\ifauditmarks\ifmmode\mkern2mu\else\hskip0.12em\relax\fi\fi}
\begin{document}
\title[The capillary $L_p$ dual Minkowski problem]
 {The capillary $L_p$ dual Minkowski problem for $p>q$ in higher dimensions}

\author{ Ya Gao$^{\dagger,\ast}$}

\address{
 $^{\dagger}$School of Mathematics and Statistics, South-Central Minzu University, Wuhan 430074, China. }

\email{Echo-gaoya@outlook.com}

\thanks{$\ast$ Corresponding author}

\date{}
%\maketitle
\begin{abstract}

We prove that the capillary $L_p$ dual Minkowski problem in
$\mathbb{R}^{n+1}$, $n\geq3$, admits a unique positive smooth strictly
convex solution for any prescribed positive smooth function $f$ when
$p>q$ and the contact angle satisfies $\theta\in(0,\pi/2)$.
Earlier results for $p>q$ either required the additional condition
$q\leq1$ or were restricted to $n=2$.
We reduce the problem to a Monge--Amp\`ere type equation with
a Robin boundary condition on a spherical cap. By constructing
a new auxiliary function for the global $C^2$ estimate,
we overcome these limitations and completely resolve the case of $p>q$ in
higher dimensions, i.e., prove existence and uniqueness throughout the case of $p>q$, without any further restriction on $p$ or $q$.

\end{abstract}

\maketitle {\it \small{{\bf Keywords}: $L_p$ dual Minkowski problem, Capillary hypersurfaces, Monge-Amp\`ere equation, Robin boundary condition.}

{{\bf MSC 2020}: Primary 53C65, 35J66. Secondary 53C42, 53C45, 35J60.}}

\section{Introduction}
The classical Minkowski problem is a fundamental problem in the Brunn--Minkowski theory of convex geometry. It asks one to determine a closed convex body from its surface area measure on the unit sphere.  Minkowski \cite{Min97,Min03} established the foundational results for convex polyhedra, and Aleksandrov \cite{Ale38,Ale39,Ale56} developed the existence and uniqueness theory for general convex bodies. For the regularity theory, see \cite{Lew38,Nir53,Pog52,Pog78,CY76}.

The $L_p$ Brunn--Minkowski theory provides a natural extension of this problem. Building on the $L_p$ addition of convex bodies, Lutwak \cite{Lut.JDG.38-1993.131,Lut.Adv.118-1996.244} introduced the $L_p$ surface area measure and the associated $L_p$-Minkowski problem. 
Further developments in the logarithmic, centroaffine and other $L_p$ problems can be found in \cite{BLYZ.JAMS.26-2013.831,CLZ.TAMS.371-2019.2623,CW.Adv.205-2006.33,HLYZ.DCG.33-2005.699,JLW.JFA.274-2018.826,JLZ.CVPDE.55-2016.41,LYZ.TAMS.356-2004.4359,Zhu.Adv.262-2014.909,Zhu.JDG.101-2015.159}.
A second development comes from the dual Brunn--Minkowski theory. Huang, Lutwak, Yang and Zhang \cite{HLYZ.Acta.216-2016.325} introduced dual curvature measures, and the
existence results and flow formulations can be found in \cite{Zha.CVPDE.56-2017.18,BHP.JDG.109-2018.411,HP.Adv.323-2018.114,LSW.JEMSJ.22-2020.893,Zha.JDG.110-2018.543}.

Later, Lutwak, Yang and Zhang \cite{LYZ18} unified these two extensions by introducing the $L_p$ dual curvature measures.
In the smooth setting, the associated $L_p$ dual Minkowski problem is equivalent to solving
\[
    \det(h_{ij}+h\delta_{ij})
    =fh^{p-1}
      \left(h^2+|\nabla h|^2\right)^{\frac{n+1-q}{2}}
    \qquad\text{on }\uS,
\]
where $p,q\in\R$, $f$ is a given positive smooth function on $\uS$, $h>0$ is the support function of a convex body in $\R^{n+1}$, and the matrix $(h_{ij}+h\delta_{ij})$ is
positive definite. Here $h_{ij}$ denotes the components of the covariant Hessian in a local orthonormal frame on $\uS$, and $\nabla h$ is the gradient in the unit spherical metric.
 The choice {$q=n+1$} recovers the {$L_p$}-Minkowski problem, while {$p=0$} gives the dual Minkowski problem. In particular, Huang and Zhao \cite{HZ.Adv.332-2018.57} proved that,
for $p>q$, the $L_p$ dual Minkowski problem admits a unique
positive smooth strictly convex solution for any prescribed positive
smooth function on $\uS$.
For further developments, see
\cite{CHZ.MA.373-2019.953,BF.JDE.266-2019.7980,CCL,HLYZ.JDG.110-2018.1}.

\subsection{The capillary $L_p$ dual Minkowski problem}
Let $\{E_i\}_{i=1}^{n+1}$ be the standard orthonormal basis of $\R^{n+1}$, and 
\[
    \R^{n+1}_{+}=\{x\in\R^{n+1}:\langle x,E_{n+1}\rangle>0\},
    \qquad e=-E_{n+1}.
\]
We call a hypersurface $\Sigma$ in $\overline{\mathbb{R}^{n+1}_+}$ with boundary $\partial\Sigma \subset \partial\mathbb{R}^{n+1}_{+}$ {\emph{capillary}} if it intersects with $\partial\mathbb{R}^{n+1}_{+}$ at a constant contact angle $\theta\in(0,\pi)$.

Mei, Wang and Weng \cite{MWW25a} first posed and solved the 
Minkowski problem for smooth capillary hypersurfaces (i.e., the capillary Minkowski problem) for {$\theta\in(0,\pi/2]$};
their subsequent work \cite{MWW25b} treated the capillary $L_p$-Minkowski problem for $p>1$.
Hu, Ivaki and Scheuer \cite{HIS25} studied the capillary Christoffel--Minkowski problem in the even strictly convex case;
Hu and Ivaki \cite{HI25} obtained even smooth strictly convex solutions to the capillary $L_p$-Minkowski problem for
$-(n+1)<p<1$ and $\theta\in(0,\pi/2)$,
with a positive smooth even function. Wang and Zhu \cite{WZ25} proposed an Orlicz extension and studied volume-normalized even solutions under their structural assumptions on the Orlicz function.
Curvature flows provide a complementary approach. Mei, Wang and Weng \cite{MWW25c} introduced the capillary Gauss curvature flow. Hu, Hu and Ivaki \cite{HHI25} studied anisotropic capillary $L_p$ Minkowski flows.

The $L_p$ dual Minkowski problem for capillary hypersurfaces
asks whether, for a prescribed positive smooth function $f$
on the spherical cap $C_\theta$, there exists a smooth strictly convex capillary
hypersurface $\Sigma$ such that
\[
d\widetilde C^c_{p,q}
=\frac{1}{n+1}fl\,d\sigma\footnote{See Section \ref{prelim capillary problem} for details. }.
\]
This problem is
equivalent, for a prescribed positive smooth function $f$ on $C_{\theta}$, to finding a positive strictly convex solution of the following Monge--Amp\`ere equation with Robin boundary condition
\begin{equation}\label{Eq}
\left\{
\begin{aligned}
    \det(\nabla^2h+h\sigma)
       &=fh^{p-1}\left(h^2+|\nabla h|^2\right)^{\frac{n+1-q}{2}}
       &&\text{in }C_\theta,\\
    \nabla_\mu h&=\cot\theta\,h
       &&\text{on }\partial C_\theta,
\end{aligned}
\right.
\end{equation}
where $h$ is the support function on $C_{\theta}$, $\sigma$ is the unit spherical metric on $C_{\theta}$ and $\mu$ is the outward unit normal of $\partial C_{\theta}$. Write $A=\nabla^2 h +h\sigma$; we call $h$ strictly convex when $A>0$. 

In view of Huang and Zhao's \cite{HZ.Adv.332-2018.57} complete resolution of the $L_p$ dual Minkowski problem for $p>q$ in the closed setting, together with the recent results for capillary hypersurfaces,
it is natural to ask the following question:

\begin{quote}
\emph{For $p>q$, does the $L_p$ dual Minkowski problem for
capillary hypersurfaces admit a unique smooth strictly convex
solution?}
\end{quote}

In our previous work \cite{GY26v2}, we formulated the $L_p$ dual Minkowski problem
for capillary hypersurfaces and established existence and uniqueness for any positive smooth function when
$p>q, q\leq1, \theta\in(0,\tfrac{\pi}{2})$.
 In the second version of their preprint, Hu and Yang \cite{HY26} removed the condition $q\leq1$ in $\R^3$, that is, for $n=2$. Inspired by Hu and Yang \cite{HY26}, we construct a new auxiliary function by adding the term $\log\sigma_{n-1}(A)$ for the $C^2$ estimate. We completely resolve the capillary $L_p$ dual Minkowski problem for $p>q$ in higher dimensions.

\subsection{Main result}
We prove the following theorem for $p>q$ in $\mathbb{R}^{n+1}$, $n\geq3$.

\begin{theorem}\label{main1.1}
Let $n\geq 3$, $p>q$ and $\theta\in (0, \frac{\pi}{2})$. For any positive smooth function $f$ defined on $C_{\theta}$, there exists a unique positive smooth strictly convex solution $h$ of Eq. \eqref{Eq}.
\end{theorem}

Theorem \ref{main1.1} extends the parameter range of
\cite{GY26v2} and gives a higher-dimensional counterpart of
the result in \cite{HY26}. The contribution of the
present proof is a different global $C^2$ estimate based on
coupling the trace and the inverse trace in the auxiliary
function.
The condition $q\leq1$ in \cite{GY26v2} entered the interior $C^2$ estimate, rather than the lower-order estimates or the uniqueness argument.
Let $\sigma_k(A)$ be the $k$-th elementary symmetric function of the positive eigenvalues of $A$.
Hu and Yang's auxiliary function has the form
\[
    \mathfrak Q=\log\sigma_1(A)+\mathfrak c_1h
                    +\mathfrak c_2|\nabla h|^2.
\]
Their boundary argument for $n=2$ uses the fact that the single tangential eigenvalue is determined by the determinant and the normal eigenvalue. In higher dimensions, the same information controls only the product of the tangential eigenvalues, not their sum. A different boundary estimate is therefore needed. 

The construction of our auxiliary function for the $C^2$ estimate is
inspired by that of Hu and Yang \cite{HY26}. We add the term
$\log\sigma_{n-1}(A)$ to control the trace and the inverse trace
simultaneously in higher dimensions, i.e., 
we consider
\[
    Q=\log\sigma_1(A)+ah+b|\nabla h|^2+\log\sigma_{n-1}(A),
    \qquad b>0,\quad a<0.
\]
Indeed, $\sigma_{n-1}(A)=(\det A)\operatorname{tr}(A^{-1})$.
At a boundary maximum, the additional term produces a
negative sum of squares of differences between reciprocal
principal radii. This term controls a large inverse trace
when the normal eigenvalue is bounded below; the complementary
case is handled by the term $ah$.
At an interior maximum, the Codazzi equation and
Lemma \ref{prelim algebraic inequalities} control the
additional third-derivative terms. Thus the global $C^2$
estimate holds without the restriction $q\leq1$.

\setcounter{prop}{\value{theorem}}
\begin{remark}[]\label{intro related work}

In the third version of the same preprint, Hu, Hu and Yang
\cite{HHY26v3} prove existence and uniqueness for arbitrary positive
smooth $f$ when $p>q$, $n\geq2$, and $\theta\in(0,\pi/2)$.
Their result therefore includes the conclusion of
Theorem~\ref{main1.1}. Their auxiliary function has, in our
notation, the form
\[
    \log\sigma_1(A)+c_1h+c_2|\nabla u|^2,
    \qquad u=h/l.
\]

The distinction lies in the $C^2$ estimate. In
\cite[Lemma 5.4]{HHY26v3}, the homogeneous Neumann condition
for $u=h/l$ is used to exclude boundary maxima of the
auxiliary function. Our argument instead retains
$|\nabla h|^2$ and estimates a boundary maximum directly.

While completing the first version of this paper, we learned that
Hu, Hu and Yang had also solved the problem in higher dimensions. We mentioned their result in a remark in our first arXiv version
before their revised preprint was posted. The present second version
adds \cite{HHY26v3} and clarifies its relation to \cite{HY26}
and to the present work.
\end{remark}

\begin{remark}[]\label{intro scope}
If $h$ solves Eq.\ \eqref{Eq} with density $f$, then $rh$, $r>0$, has density $r^{q-p}f$. When $p=q$, uniqueness without a normalization therefore cannot hold. In the range $p>q$, the sign of $p-q$ gives both the positive lower-order bounds and the negative zeroth-order coefficient in the logarithmic formulation.
\end{remark}

\subsection{Outline of the proof}
We prove Theorem \ref{main1.1} by establishing global a priori estimates and then applying the continuity method. 

First, we use $u=h/l$ to obtain positive upper and lower bounds. The identity $\nabla^2l+l\sigma=\sigma$ gives
\[
    A_{ij}=lu_{ij}+u\sigma_{ij}+l_i u_j+l_j u_i,
    \qquad \nabla_\mu u=0.
\]
At an extremum of $u$ on the compact cap, $\nabla u=0$ and the Hessian has the appropriate sign. The same conclusion holds at a boundary extremum, using the homogeneous Neumann condition and its tangential derivatives. Set
\[
    \beta=n+1-q,\qquad
    \psi=l^2+|\nabla l|^2=2l-\sin^2\theta.
\]
Comparison with $u\sigma$ at these extrema yields
\[
    \min_{C_\theta}\bigl(f^{-1}l^{1-p}\psi^{-\beta/2}\bigr)
       \leq u^{p-q}
       \leq \max_{C_\theta}\bigl(f^{-1}l^{1-p}\psi^{-\beta/2}\bigr).
\]
Since $p>q$ and $l,\psi$ are positive, this gives $0<m\leq h\leq M$; see Lemma \ref{C0 estimate}. The gradient estimate in Lemma \ref{c1-1} uses only strict convexity and the boundary condition. Applying the maximum principle to $P=h^2+|\nabla h|^2$, and using $h_{\alpha\mu}=0$ along the boundary for tangential indices $\alpha$, gives
\[
    \max_{C_\theta}|\nabla h|
       \leq (1+\cot^2\theta)^{1/2}\Vert h\Vert_{C^0(C_\theta)}.
\]
The equation then bounds {$\det A$} above and below by positive constants.

The main estimate is Theorem \ref{c2 estimate}. Let $\sigma_k(A)$ be the $k$-th elementary symmetric function of the positive eigenvalues of $A$. We consider
\[
    Q=\log\sigma_1(A)+ah+b|\nabla h|^2+\log\sigma_{n-1}(A),
    \qquad b>0,\quad a<0.
\]
The maximality of $Q$, together with the determinant bounds,
then yields
\[
    C^{-1}\sigma\leq A\leq C\sigma
    \qquad\text{on }C_\theta,
\]
which gives the global $C^2$ estimate. The equation is
therefore uniformly elliptic, and higher-order estimates
follow; see Corollary \ref{c2 theorem}.

\begin{remark}[]\label{intro angle dimension}
The $C^0$ and $C^1$ estimates hold for $\theta\in(0,\pi)$. The acute-angle assumption is used in the $C^2$ estimate, in particular through the positive factor $\cot\theta$ in the boundary estimate. At $\theta=\pi/2$ this coercive term vanishes, so the endpoint is not a consequence of the present proof. When $n=2$, $\sigma_{n-1}=\sigma_1$, and our auxiliary function becomes $2\log\sigma_1+ah+b|\nabla h|^2$. Dividing by two and rescaling $a,b$ gives the form used in Hu and Yang \cite{HY26}. The additional symmetric polynomial has distinct content in dimensions $n\geq3$.
\end{remark}

For existence, we apply the continuity method in the logarithmic
variable $v=\log h$, starting from $h=l$. Since $p>q$, the
linearized problem is invertible, giving openness, while
the uniform a priori estimates yield closedness.
Uniqueness follows from the concavity of $\log\det$,
the maximum principle and the Hopf boundary lemma.

This paper is organized as follows. In Section \ref{se2}, we recall the capillary curvature measure and give some important formulas. In Section \ref{se3}, we establish the positive $C^0$ bounds and the gradient estimate. Section \ref{se4} contains the global $C^2$ estimate and its regularity consequence. In Section \ref{se5}, we carry out the continuity argument in logarithmic variables and prove uniqueness, completing the proof of Theorem \ref{main1.1}.

\medskip
\noindent\textbf{Acknowledgements.}
The author is deeply grateful to Prof. Zhenlei Zhang for his
guidance and generous support, and also sincerely thanks
Prof. Yingxiang Hu for his valuable comments and suggestions.

\section{Preliminaries } \label{se2}
We use the notation from Section 1. We first describe the capillary $L_p$ dual Minkowski problem, the radial maps and
the measure formulation, and then prove the differential identities needed
for the a priori estimates.

\subsection{The capillary $L_p$ dual Minkowski problem}\label{prelim capillary problem}
We consider a smooth, compact, properly embedded capillary hypersurface $\Sigma\subset\overline{\R^{n+1}_{+}}$ with
\[
    \operatorname{int}(\Sigma)\subset\R^{n+1}_{+},
    \qquad \partial\Sigma\subset\partial\R^{n+1}_{+}.
\]
For a convex capillary hypersurface $\Sigma$, let
$\widehat{\Sigma}$ be the bounded closed region in
$\overline{\R^{n+1}_{+}}$ enclosed by $\Sigma$ and
$\partial\R^{n+1}_{+}$, and denote
\[
    \widehat{\partial\Sigma}
    :=\partial\widehat{\Sigma}\setminus\Sigma
    \subset\partial\R^{n+1}_{+}.
\]
Let $\nu$ be the unit outward normal of $\Sigma$ with respect
to $\widehat{\Sigma}$, i.e., the Gauss map of $\Sigma$.
The contact angle condition is given by
\[
    \langle\nu,e\rangle=\cos(\pi-\theta)=-\cos\theta
    \qquad\text{along }\partial\Sigma.
\]
We denote by $\mathcal K_\theta$ the family of capillary convex bodies, and by {$\mathcal K^\circ_\theta$} those for which the origin lies in the relative interior
of the flat part of their boundary.

For a strictly convex capillary hypersurface, the usual Gauss map identifies $\Sigma$ with
\[
    \S^n_\theta=\{z\in\uS:\langle z,E_{n+1}\rangle\geq\cos\theta\}.
\]
The capillary Gauss map $\widetilde\nu=T\circ\nu$, where $T(z)=z+\cos\theta\,e$, is a diffeomorphism from $\Sigma$ to $C_\theta$.
Here
\[
    C_\theta=T(\S^n_\theta)=\{\xi\in\overline{\R^{n+1}_{+}}:|\xi-\cos\theta\,e|=1\}.
\]
The cap $C_\theta$ includes its boundary. We use $\sigma$ for its unit spherical metric and $\nabla$ for the associated Levi-Civita connection; norms and determinants are taken with respect to $\sigma$. The vector $\mu$ is the outward unit normal of $\partial C_\theta$ within $C_\theta$. 

Writing $X=\widetilde\nu^{-1}$, we regard the usual support function as a function on $C_\theta$:
\[
    h(\xi)=\langle X(\xi),\xi-\cos\theta\,e\rangle.
\]
If $\widehat\Sigma\in\mathcal K^\circ_\theta$, then $h>0$.
Under the parametrization
$X=\widetilde\nu^{-1}:C_\theta\to\Sigma$
by the inverse capillary Gauss map, we have
\[
    X=\nabla h+h(\xi-\cos\theta\,e),
    \qquad K_G(X(\xi))=(\det A)^{-1},
\]
where {$K_G$} is the Gauss--Kronecker curvature, and {$A$} is the
second fundamental form of {$\Sigma$} pulled back by {$X$}, with the
convention that it is positive on strictly convex hypersurfaces.
Its eigenvalues with respect to {$\sigma$} are the principal radii of curvature.
We write {$d\sigma$} for the area measure of {$\sigma$}.

Set
\[
   l(\xi)=\langle\xi,\xi-\cos\theta\,e\rangle=\sin^2\theta+\cos\theta\langle\xi,e\rangle,\qquad u=\frac{h}{l}.
\]
The support function $l$ of the unit cap is positive on the compact cap for every $\theta\in(0,\pi)$ and satisfies
$\nabla^2l+l\sigma=\sigma$.
At $\xi\in\partial C_\theta$,
\[
    \mu(\xi)=\frac{e+\cos\theta(\xi-\cos\theta\,e)}{\sin\theta}.
\]
Since $X(\partial C_\theta)\subset\partial\R^{n+1}_{+}$, the identity $0=\langle X,e\rangle=\sin\theta\,\nabla_\mu h-\cos\theta\,h$ gives the Robin condition for $h$.
Along $\partial C_\theta$, we have
\[
    \nabla_\mu l=\cot\theta\,l,
    \qquad \nabla_\mu h=\cot\theta\,h.
\]
Consequently, the capillary support function $u$ satisfies the homogeneous Neumann condition $\nabla_\mu u=0$. This distinction between the Robin condition for $h$ and the Neumann condition for $u$ is important in the estimates.

Evenness on $C_\theta$ means invariance under horizontal reflection, as in \cite{MWW25b}:
\[
    f(-\xi_1,\ldots,-\xi_n,\xi_{n+1})=f(\xi_1,\ldots,\xi_n,\xi_{n+1}).
\]
The same reflection defines symmetry for capillary hypersurfaces and convex bodies; it is not antipodal symmetry on $\uS$.

Assume now that $\widehat\Sigma\in\mathcal K^\circ_\theta$.
On the closed upper hemisphere
$\uS_+=\{\gamma\in\uS:\langle\gamma,E_{n+1}\rangle\geq0\}$, the radial
function and radial map are given by
\[
    \rho(\gamma)=\max\{t\geq0:t\gamma\in\widehat\Sigma\},
    \qquad r(\gamma)=\rho(\gamma)\gamma\in\Sigma.
\]
Since the origin lies in the relative interior of the flat base, 
$\rho>0$ on $\uS_+$ and $r$ maps $\partial\uS_+$ onto
$\partial\Sigma$.
The radial Gauss map is $\alpha=\nu\circ r:\uS_+\to\S^n_\theta$, with
inverse $\alpha^*=\alpha^{-1}$. As in \cite{GY26v2},
\[
    \widetilde\alpha=T\circ\alpha\circ r^{-1}
       =\widetilde\nu:\Sigma\to C_\theta,
    \qquad \widetilde\alpha^{-1}=X.
\]
Thus the radial function, expressed in the parameter $\xi\in C_\theta$,
satisfies
\[
    \rho\bigl(\alpha^*(T^{-1}\xi)\bigr)^2
       =|X(\xi)|^2=h^2+|\nabla h|^2.
\]

For $p,q\in\R$ and a Borel set $\omega\subset C_\theta$, the capillary $(p,q)$-th dual curvature measure of $\widehat\Sigma$, as in \cite{GY26v2}, can be written as
\[
    \widetilde C^c_{p,q}(\widehat\Sigma,\omega)
    =\frac{1}{n+1}\int_{X(\omega)}
      \bigl(1+\cos\theta\langle\nu(y),e\rangle\bigr)
      \langle y,\nu(y)\rangle^{1-p}
      |y|^{q-n-1}\,dS_\Sigma(y),
\]
where $y\in\Sigma$ is the integration variable and $dS_\Sigma$
is the Euclidean area measure on $\Sigma$.
 Indeed,
differentiating the representation of $X$ gives
$dX(e_i)=\sum_j A_{ij}e_j$ in a spherical orthonormal frame. Consequently,
\[
    g_{ij}=\sum_k A_{ik}A_{jk},
    \qquad X^*(dS_\Sigma)=\det A\,d\sigma,
\]
where $g=X^*g_{\mathbb R^{n+1}}$ is the induced metric. The change of variables {$y=X(\xi)$} yields
\[
    d\widetilde C^c_{p,q}
       =\frac{1}{n+1}lh^{1-p}
          \left(h^2+|\nabla h|^2\right)^{\frac{q-n-1}{2}}
          \det A\,d\sigma.
\]

As in \cite{GY26v2}, we formulate the smooth problem as follows.

\smallskip
\noindent\textbf{Capillary $L_p$ dual Minkowski problem.}
\emph{Given real numbers $p,q$ and a positive smooth function $f$ on $C_\theta$, does there exist a capillary body {$\widehat\Sigma\in\mathcal K^\circ_\theta$} with smooth strictly convex curved boundary {$\Sigma$} such that}
\[
    d\widetilde C^c_{p,q}=\frac{1}{n+1}f l\,d\sigma\,?
\]

By the support-function representation, cf.\ \cite[Proposition 2.4]{MWW25a}, this is equivalent to finding a positive strictly convex solution of Eq.\ \eqref{Eq}.

The factor $1/(n+1)$ in the prescribed measure is a normalization: prescribing $g_0 l\,d\sigma$ for a positive smooth function $g_0$ instead gives Eq.\ \eqref{Eq} with $f=(n+1)g_0$. When $q=n+1$, Eq.\ \eqref{Eq} is precisely the capillary $L_p$-Minkowski equation. At the level of measures, one should also keep track of the capillary weight. With the convention $dS_p^c=l^ph^{1-p}\det A\,d\sigma$ in \cite{MWW25b}, we have
\[
    d\widetilde C^c_{p,n+1}
      =\frac{l^{1-p}}{n+1}\,dS_p^c.
\]
Thus the two prescribed-measure formulations differ by a fixed positive smooth weight, which can be absorbed into the prescribed function. In particular, $p=1$, $q=n+1$ recovers the capillary Minkowski equation, although this pair of parameters is outside the range $p>q$ considered here.

\subsection{Basic properties and differential identities}
The relation $u=h/l$ also has the geometric interpretation
\begin{equation}\label{e1}
    X(\xi)=u(\xi)\xi+s',\qquad s'\in T_{X(\xi)}\Sigma,
\end{equation}
since $\langle\xi,\nu(X(\xi))\rangle=l(\xi)$.
We take all covariant derivatives in the spherical metric $\sigma$ and
use orthonormal frames. At the boundary, $e_n=\mu$.

\begin{proposition}\label{pro}
The support function $h$ and the capillary support function $u=l^{-1}h$
satisfy the following conditions on $\partial C_\theta$:
\begin{equation}\label{pro1}
    h_{kn}=0,
\end{equation}
which is equivalent to
\begin{equation}\label{pro2}
    u_{kn}=-\cot\theta\,u_k,
\end{equation}
where $k=1,2,\ldots,n-1$.
\end{proposition}
\begin{proof}
Since $\nabla_{e_k}\mu=\cot\theta\,e_k$ on $\partial C_\theta$, the
conditions $h_n=\cot\theta\,h$ and $u_n=0$ give
\[
\begin{aligned}
    h_{kn}&=e_k(h_n)-\langle\nabla h,\nabla_{e_k}\mu\rangle
           =\cot\theta\,h_k-\cot\theta\,h_k=0,\\
    u_{kn}&=e_k(u_n)-\langle\nabla u,\nabla_{e_k}\mu\rangle
           =-\cot\theta\,u_k.
\end{aligned}
\]
\end{proof}

The following identities hold for any $h\in C^4(C_\theta)$ with
$A=\nabla^2h+h\sigma>0$. Denote
\[
    P=h^2+|\nabla h|^2,\qquad F=\log\det A,
    \qquad S=\tr(A^{-1}).
\]
At a point where $A$ is diagonal, denote its eigenvalues by $\lambda_i>0$,
and let $\sigma_k(A)$ be their $k$-th elementary symmetric function.
A comma denotes covariant differentiation; in particular, {$h_{ij,k}=\nabla_k\nabla_{ij}^2h$}. Indices in explicit sums range from {$1$} to {$n$} unless otherwise specified.

\begin{lemma}\label{prelim differential}
The tensor $A$ satisfies
\begin{equation}\label{prelim Codazzi}
    A_{ij,k}=A_{ik,j},\qquad
    A_{ii,jj}=A_{jj,ii}+A_{ii}-A_{jj},
\end{equation}
where $i,j$ are fixed in the second identity. Moreover,
\begin{equation}\label{prelim radial derivatives}
\begin{aligned}
    P_i&=2\sum_k A_{ik}h_k,\\
    P_{ij}&=2\sum_k h_kA_{ij,k}
              +2\sum_k A_{ik}A_{kj}-2hA_{ij}.
\end{aligned}
\end{equation}
\end{lemma}
\begin{proof}
The spherical commutation formula
$h_{ij,k}-h_{ik,j}=h_j\delta_{ik}-h_k\delta_{ij}$ gives the Codazzi
identity for $A$. Commuting one more derivative gives
\[
    A_{ii,jj}-A_{jj,ii}=A_{ij,ij}-A_{ij,ji}=A_{ii}-A_{jj}.
\]
Differentiating $P$ gives its first derivative in
\eqref{prelim radial derivatives}. Differentiating again and using
$A_{ik,j}=A_{ij,k}$ and $h_{ij}=A_{ij}-h\delta_{ij}$ gives the second.
\end{proof}

We next give two algebraic inequalities for the interior {$C^2$} estimate.

\begin{lemma}\label{prelim algebraic inequalities}
Let $\lambda_1,\ldots,\lambda_n>0$, and let $(\eta_{ijk})_{i,j,k=1}^n$
be symmetric in all three indices. Set
$\sigma_1=\sum_i\lambda_i$ and $S=\sum_i\lambda_i^{-1}$. Then
\begin{equation}\label{prelim algebraic trace}
    \sum_{i,j,k}\frac{\eta_{ijk}^2}{\lambda_i\lambda_j}
       \geq\frac{1}{\sigma_1}\sum_k\frac{1}{\lambda_k}
                    \left(\sum_i\eta_{iik}\right)^2
\end{equation}
and
\begin{equation}\label{prelim algebraic inverse trace}
    \sum_{i,j,k}\left(\frac{2}{\lambda_i^2\lambda_j\lambda_k}
                       -\frac{1}{\lambda_i\lambda_j\lambda_k^2}\right)
                       \eta_{ijk}^2
       \geq\frac{1}{S}\sum_k\frac{1}{\lambda_k}
                    \left(\sum_i\frac{\eta_{iik}}{\lambda_i^2}\right)^2.
\end{equation}
\end{lemma}
\begin{proof}
For each fixed $k$, the Cauchy--Schwarz inequality gives
\[
    \left(\sum_i\eta_{iik}\right)^2
    =\left(\sum_i\sqrt{\lambda_i}\,
                  \frac{\eta_{iik}}{\sqrt{\lambda_i}}\right)^2
    \leq \sigma_1\sum_i\frac{\eta_{iik}^2}{\lambda_i}.
\]
So we have
\[
\begin{aligned}
    \frac{1}{\sigma_1}\sum_k\frac{1}{\lambda_k}
        \left(\sum_i\eta_{iik}\right)^2
    &\leq \sum_{i,k}\frac{\eta_{iik}^2}{\lambda_i\lambda_k}\\
    &\leq \sum_{i,j,k}\frac{\eta_{ijk}^2}{\lambda_i\lambda_k}
     = \sum_{i,j,k}\frac{\eta_{ijk}^2}{\lambda_i\lambda_j}.
\end{aligned}
\]
The second inequality follows by adding the nonnegative terms
with $i\neq j$. The last equality follows by interchanging
$j$ and $k$ and using the symmetry of $\eta_{ijk}$.
This proves \eqref{prelim algebraic trace}.

For the second inequality, the Cauchy--Schwarz inequality gives
\[
    \left(\sum_i\frac{\eta_{iik}}{\lambda_i^2}\right)^2
    =\left(\sum_i\frac{1}{\sqrt{\lambda_i}}\,
                  \frac{\eta_{iik}}{\lambda_i^{3/2}}\right)^2
    \leq S\sum_i\frac{\eta_{iik}^2}{\lambda_i^3}.
\]
On the other hand, interchanging $i$ and $k$ and using symmetry, we have
\[
    \sum_{i,j,k}\frac{\eta_{ijk}^2}
                       {\lambda_i\lambda_j\lambda_k^2}
    =\sum_{i,j,k}\frac{\eta_{kji}^2}
                       {\lambda_i^2\lambda_j\lambda_k}
    =\sum_{i,j,k}\frac{\eta_{ijk}^2}
                       {\lambda_i^2\lambda_j\lambda_k}.
\]
So
\[
\begin{aligned}
    &\sum_{i,j,k}
      \left(\frac{2}{\lambda_i^2\lambda_j\lambda_k}
            -\frac{1}{\lambda_i\lambda_j\lambda_k^2}\right)
      \eta_{ijk}^2\\
    &\qquad
      =\sum_{i,j,k}\frac{\eta_{ijk}^2}
                            {\lambda_i^2\lambda_j\lambda_k}\\
    &\qquad
      \geq\sum_{i,k}\frac{\eta_{iik}^2}{\lambda_i^3\lambda_k}\\
    &\qquad
      \geq\frac{1}{S}\sum_k\frac{1}{\lambda_k}
               \left(\sum_i\frac{\eta_{iik}}{\lambda_i^2}\right)^2.
\end{aligned}
\]
This proves \eqref{prelim algebraic inverse trace} and completes the proof.
\end{proof}

By \eqref{prelim Codazzi}, Lemma \ref{prelim algebraic inequalities}
applies to $\eta_{ijk}=A_{ij,k}$.
With these preliminaries, we now turn to the a priori estimates for
Eq. \eqref{Eq}. We first establish positive lower and upper bounds for
$h$ and a gradient estimate in Section \ref{se3}, and then derive the
global $C^2$ estimate in Section \ref{se4}.

\section{ $C^0$ and $C^1$ estimates} \label{se3}
We establish positive {$C^0$} bounds and a gradient estimate for solutions of Eq. \eqref{Eq}.

\begin{lemma}[{\bf $C^0$ estimate}]\label{C0 estimate}
    Let $p>q$ and $\theta\in(0, \pi)$. Suppose $h$ is a positive smooth strictly convex solution to Eq. \eqref{Eq}, that is, $A=\nabla^2h+h\sigma>0$ on $C_{\theta}$. Then we have 
    \begin{equation}\label{a1}
       \min\{ 1-\cos\theta, \sin^2\theta\}^{p-q}\cdot C_2
        \leq h^{p-q}\leq
       \max\{ 1-\cos\theta, \sin^2\theta\}^{p-q}\cdot C_1 ,
    \end{equation}
    where the positive constants {$C_1$} and {$C_2$} depend only on $n$, $p$, $q$, $\theta$, $\min_{C_{\theta}}f$ and $\Vert f\Vert_{C^0(C_{\theta})}$.
    In particular, $h$ has positive lower and upper bounds depending only on
    $n,p,q,\theta$, $\min_{C_{\theta}}f$ and $\Vert f\Vert_{C^0(C_{\theta})}$.
\end{lemma}

\begin{proof}
    Let
    \[
        u=l^{-1}h.
    \]
    Since $h$ solves Eq. \eqref{Eq}, the function $u$ satisfies
    \begin{equation}\label{a2}
    \left\{
    \begin{aligned}
    &\det\left(l\nabla^2 u+\cos\theta\bigl(\nabla u\otimes e^T
    +e^T\otimes\nabla u\bigr)+u\sigma\right)\\
    &\qquad\qquad
    =f(ul)^{p-1}\left((ul)^2+|\nabla(ul)|^2\right)^{\frac{n+1-q}{2}}
    &&\text{in }C_{\theta},\\
    &\nabla_{\mu}u=0&&\text{on }\partial C_{\theta},
    \end{aligned}
    \right.
    \end{equation}
    where $e^T$ is the tangential part of $e$ on $C_{\theta}$, and we used that $\nabla_{\mu} l = \cot\theta\cdot l$ on $\partial C_{\theta}$.
    Let $\beta = n+1-q$ and  $\zeta=\xi-\cos\theta e$. Then $|\zeta|=1$ on $C_{\theta}$ and
    \[
        l=\sin^2\theta+\cos\theta\langle\xi,e\rangle
        =1+\cos\theta\langle\zeta,e\rangle.
    \]
    Since $-1\leq\langle\zeta,e\rangle\leq-\cos\theta$ on $C_{\theta}$, we obtain    \[
        \min\{1-\cos\theta, \sin^2\theta\}\leq l\leq \max\{1-\cos\theta, \sin^2\theta\}.
    \]
    Moreover,
    \[
         |\nabla\langle\zeta,e\rangle|^2=1-\langle\zeta,e\rangle^2.
    \]
    Setting {$\psi=l^2+|\nabla l|^2$}, we obtain
    \[
        \psi
        =(1+\cos\theta\langle\zeta,e\rangle)^2
        +\cos^2\theta\bigl(1-\langle\zeta,e\rangle^2\bigr)
        =2l-\sin^2\theta,
    \]
    hence
    \[
      \min\{  (1-\cos\theta)^2, \sin^2\theta\} \leq \psi\leq \max\{(1-\cos\theta)^2, \sin^2\theta\}.
    \]
    Since $h=ul$ and $\nabla^2l+l\sigma=\sigma$, we have
    \[
        A_{ij}=l u_{ij}+u\sigma_{ij}+l_i u_j+l_j u_i.
    \]
    We also have $\nabla_{\mu}h=\cot\theta\cdot h$ and $\nabla_{\mu}l=\cot\theta\cdot l$ on
    $\partial C_{\theta}$, so
    \[
        \nabla_{\mu}u=0\qquad\text{on }\partial C_{\theta}.
    \]

    At any interior minimum (or maximum) point of $u$ we have $\nabla u=0$ and $\nabla^2u\geq 0$ (or $\nabla^2 u\leq 0$). At a boundary extremum, the tangential derivatives vanish and
    {$u_\mu=0$}. Proposition \ref{pro} therefore gives
    {$u_{\alpha\mu}=0$} for {$\alpha=1,\ldots,n-1$}.
    Since $u_\mu=0$, the one-sided second derivative test along
the geodesic with initial velocity $-\mu$ gives
$u_{\mu\mu}\leq0$ at a boundary maximum and
$u_{\mu\mu}\geq0$ at a boundary minimum.
    Thus the Hessian has the same semidefinite sign at a boundary extremum.

    At an extremum point of $u$ the terms $l_i u_j+l_j u_i$ vanish, and
    $
        h^2+|\nabla h|^2=u^2\bigl(l^2+|\nabla l|^2\bigr)=u^2\psi,
    $
so Eq. \eqref{Eq} becomes
    \[
        \det(l\nabla^2u+u\sigma)
        =f u^{n+p-q}l^{p-1}\psi^{\beta/2}.
    \]

    Suppose {$u$} attains its maximum value at some point $\xi_0\in C_{\theta}$. At $\xi_0$, we obtain
    \[
        f u^{n+p-q}l^{p-1}\psi^{\beta/2}
        =\det(l\nabla^2u+u\sigma)\leq u^n,
    \]
    which gives
    \[
        u(\xi_0)^{p-q}\leq
        f(\xi_0)^{-1}l(\xi_0)^{1-p}\psi(\xi_0)^{-\beta/2}.
    \]
   Set {$G:=f^{-1}l^{1-p}\psi^{-\beta/2}$}. Then 
    $$u(\xi_0)^{p-q} \leq \max_{C_{\theta}}G=:C_1,$$
    where the positive constant $C_1$ depends on $n$, $p$, $q$, $\theta$, $\min_{C_{\theta}}f$ and $\Vert f\Vert_{C^0(C_{\theta})}$.
    Therefore, for every $\xi\in C_{\theta}$, we have
    \[
        h(\xi)^{p-q}
        =u(\xi)^{p-q}l(\xi)^{p-q}
        \leq \max\{1-\cos\theta, \sin^2\theta\}^{p-q}\cdot C_1.
    \]

    Similarly, if $u$ attains the minimum value at $\xi_1\in C_{\theta}$, there holds
    \[
        f u^{n+p-q}l^{p-1}\psi^{\beta/2}
        =\det(l\nabla^2u+u\sigma)\geq u^n,
    \]
    which gives
    \[
        u(\xi_1)^{p-q}\geq
        f(\xi_1)^{-1}l(\xi_1)^{1-p}\psi(\xi_1)^{-\beta/2}
        \geq \min_{C_{\theta}}G=:C_2,
    \]
    where the positive constant $C_2$ depends on $n$, $p$, $q$, $\theta$ and $\Vert f\Vert_{C^0(C_{\theta})}$.
   Therefore, for every $\xi\in C_{\theta}$, we have
    \[
        h(\xi)^{p-q}
        =u(\xi)^{p-q}l(\xi)^{p-q}
        \geq \min\{1-\cos\theta, \sin^2\theta\}^{p-q}\cdot C_2.
    \]
   This completes the proof.
\end{proof}

Next, we establish the $C^1$ estimate using only strict convexity
and the Robin boundary condition, independently of the
interior equation.

\begin{lemma}[{\bf $C^1$ estimate}]\label{c1-1}
    Let $p>q$ and $\theta\in (0,\pi)$. Suppose that {$h$} is a positive smooth strictly convex solution of Eq. \eqref{Eq}. Then
    \begin{equation}\label{b1}
        \max_{C_{\theta}}|\nabla h| \leq (1+\cot^2\theta)^{1/2} \Vert h\Vert_{C^{0}(C_{\theta})}.
    \end{equation}
\end{lemma}

\begin{proof}
    We consider the function
    $$P:= |\nabla h|^2 + h^2.$$
    Suppose that the function $P$  attains its maximum value at some point $\xi_{0}\in C_{\theta}$. If $\xi_0 \in C_{\theta}\backslash\partial C_{\theta}$, we have
    $$0=\nabla_{e_i}P = 2\sum_k h_{k}h_{ki} + 2hh_{i}, \qquad 1\leq i\leq n.$$
Together with the convexity of $h$, i.e., $h_{ij}+h\delta_{ij}>0$, it follows that {$\nabla h(\xi_0)=0$}. Hence {$|\nabla h|^2\leq P(\xi_0)=h(\xi_0)^2$}, which implies \eqref{b1}.

If $\xi_{0}\in \partial C_{\theta}$, we choose an orthonormal frame $\{e_i\}_{i=1}^{n}$ around $\xi_{0}\in\partial C_{\theta}$ such that $e_n = \mu$. From Proposition \ref{pro}, we know that
\begin{equation}\label{b2}
    h_{\alpha n}=0 \qquad \text{for~~any} \quad 1\leq \alpha\leq n-1.
\end{equation}
Then we have
$$0=\nabla_{e_\alpha}P = 2\sum_{i=1}^{n}h_{i}h_{i\alpha}+ 2hh_{\alpha},$$
which implies
\begin{equation}\label{b3}
    h_{\alpha}(\xi_{0}) = 0.
\end{equation}
Together with \eqref{b3} and $h_n = \cot\theta\cdot h$ on $\partial C_{\theta}$, we have
\[
\displaystyle |\nabla h(\xi)|^2\leq P(\xi_0)
      =(1+\cot^2\theta)h(\xi_0)^2
      \leq(1+\cot^2\theta)\Vert h\Vert_{C^0(C_\theta)}^2
\]
for any {$\xi\in C_\theta$}.
In conclusion, \eqref{b1} holds.
\end{proof}

\section{$C^2$ estimates}\label{se4}
In this section, we establish the global $C^2$ estimate for positive strictly
convex solutions of Eq.\ \eqref{Eq}, and then obtain the higher-order bounds
needed for the continuity method.

We recall the notation
\[
    A_{ij}=h_{ij}+h\sigma_{ij},\qquad
    P=h^2+|\nabla h|^2,\qquad
    S=\operatorname{tr}(A^{-1}),\qquad \beta=n+1-q.
\]
All covariant derivatives are taken with respect to the spherical metric
{$\sigma$}, and {$(A^{ij})$} denotes the inverse of {$(A_{ij})$}.
Recall that {$F(A)=\log\det A$} and set
{$\widetilde f=(p-1)\log h+\frac{\beta}{2}\log P+\log f$}.
Thus {$F(A)=\widetilde f$} along a solution of Eq.\ \eqref{Eq}.
In an orthonormal frame diagonalizing $A$, its eigenvalues are denoted by
$\lambda_i>0$. We write $\sigma_k(A)$ for their $k$-th elementary symmetric
function and abbreviate $\sigma_1(A)$ to $\sigma_1$. All sums run from
$1$ to $n$ unless otherwise indicated. 

\begin{theorem}[{\bf{$C^2$ estimate}}]\label{c2 estimate}
Let $n\geq3$, $p>q$, and $\theta\in(0,\frac{\pi}{2})$. Suppose that
$h\in C^4(C_\theta)$ is a positive strictly convex solution of
Eq.\ \eqref{Eq} satisfying
\[
    0<m\leq h\leq M,\qquad |\nabla h|\leq K.
\]
Then
\begin{equation}\label{d1}
    \max_{C_\theta}|\nabla^2h|\leq C,
\end{equation}
where the positive constant $C$ depends only on
$n,p,q,\theta,m,M,K,\min_{C_\theta}f$ and
$\Vert f\Vert_{C^2(C_\theta)}$.
\end{theorem}

\begin{proof}
The equation and the assumed bounds give
\[
    0<D_-\leq\det A\leq D_+,\qquad
    m^2\leq P\leq M^2+K^2,
\]
where the positive constants $D_-$ and $D_+$ depend only on
$n,p,q,m,M,K,\min_{C_\theta}f$ and $\Vert f\Vert_{C^0(C_\theta)}$.
In particular, $\sigma_1\geq nD_-^{1/n}>0$.
Consider
\begin{equation}\label{auxiliary Q}
    Q=\log\sigma_1+ah+b|\nabla h|^2+\log\sigma_{n-1}(A),
\end{equation}
where $b>0$ and $a<0$ will be chosen later. Suppose that $Q$ attains its
maximum at $\xi_0\in C_\theta$. We divide this proof into two cases: either {$\xi_0 \in C_{\theta} \backslash \partial C_{\theta}$} or {$\xi_0\in \partial C_{\theta}$}.

\noindent\textbf{Case 1. $\xi_0\in C_\theta\setminus\partial C_\theta$.}
In this case, we choose an orthonormal frame {$\{e_i\}_{i=1}^n$} around {$\xi_0$}, such that {$A_{ij} = (h_{ij}+h\delta_{ij})$} is diagonal; hence {$\frac{\partial F(A)}{\partial A_{ij}}$} and {$h_{ij}$} are also diagonal at {$\xi_0$}. We write
\[
\displaystyle F(A):=\log\det A = (p-1)\log h+\frac{\beta}{2}\log P+\log f=:\widetilde f,
\]
\[
    \log\sigma_{n-1}(A)=F+\log S,
\]
and
\[
\displaystyle F^{ij}:=\frac{\partial F(A)}{\partial A_{ij}},\qquad
 F^{ij,kl}:=\frac{\partial^2F(A)}{\partial A_{ij}\partial A_{kl}}.
\]
Then we have
\[
\displaystyle F^{ij}=A^{ij},\qquad
 F^{ij,kl}=-F^{il}F^{jk}=-A^{il}A^{jk},
\]
where {$A^{ij}$} is the inverse of {$A_{ij}$}. Note that {$F^{ij}A_{ij}=n$}. We also have
\begin{equation}\label{F's derivative}
    \begin{aligned}
F_k &= \sum_{i,j}F^{ij} A_{ij,k} = \sum_{i,j}A^{ij}A_{ij,k}=\sum_{i}\frac{A_{ii,k}}{\lambda_i}, \\ 
F_{kk}& = \sum_i \frac{A_{ii,kk}}{\lambda_i} - \sum_{i,j}\frac{A_{ij,k}^2}{\lambda_i \lambda_j},
    \end{aligned}
\end{equation}
and 
\begin{equation}\label{S's derivative}
    \begin{aligned}
S_k &= -\sum_i\frac{A_{ii,k}}{\lambda_i^2}, \\ 
S_{kk} &= -\sum_i\frac{A_{ii,kk}}{\lambda_i^2} + 2\sum_{i,j}\frac{A_{ij,k}^2}{\lambda_i^2 \lambda_j}.
    \end{aligned}
\end{equation}

Since the test function {$Q$} attains its maximum at {$\xi_0$}, we have
\begin{equation}\label{Q stationary}
    \begin{split}
    0=Q_k &=\frac{(\sigma_1)_k}{\sigma_1}
       +ah_k+2b\sum_i h_i h_{ik}
       +\left(\log \sigma_{n-1}(A)\right)_k
     \\ 
     &= \frac{\sum_i A_{ii,k}}{\sigma_1} + ah_k +2b(\lambda_k - h)h_k + (F+\log S)_k \\
     &= \sum_{i}\frac{A_{ii,k}}{\sigma_1} + ah_k +2b(\lambda_k -h)h_k + \sum_{i}\frac{A_{ii,k}}{\lambda_i} - \frac{1}{S}\sum_i \frac{A_{ii,k}}{\lambda_i^2}\\
                   &=
                   \sum_iW_iA_{ii,k}
                   +\bigl(a+2b(\lambda_k-h)\bigr)h_k,
    \end{split}
\end{equation}
where 
\[
    W_i:=\frac{1}{\sigma_1}+\frac{1}{\lambda_i}
                         -\frac{1}{S\lambda_i^2},
\]
and 
\begin{equation}\label{Q exact second derivative}
\begin{split}
0 &\geq\sum_k \frac{Q_{kk}}{\lambda_k} = \frac{1}{\sigma_1}\sum_{i,k} \frac{A_{ii,kk}}{\lambda_k} -\frac{1}{\sigma_1^2}\sum_k \frac{1}{\lambda_k}\left(\sum_{i}A_{ii,k}\right)^2 + a\sum_k \frac{h_{kk}}{\lambda_k} \\ 
&\qquad 
+2b\left(\sum_{k}\frac{h_{kk}^2}{\lambda_k} + \sum_{i,k}\frac{h_i h_{ik,k}}{\lambda_k}\right) + \sum_k \frac{F_{kk}}{\lambda_k} + \sum_k\frac{S_{kk}}{\lambda_k S} - \sum_k\frac{S_k^2}{\lambda_k S^2}\\ 
&=\frac{1}{\sigma_1}\sum_{i,k} \frac{A_{ii,kk}}{\lambda_k} -\frac{1}{\sigma_1^2}\sum_k \frac{1}{\lambda_k}\left(\sum_{i}A_{ii,k}\right)^2 + a\sum_k \frac{h_{kk}}{\lambda_k} + b\sum_k \frac{(|\nabla h|^2)_{kk}}{\lambda_k}
 \\ 
&\qquad +\sum_k\frac{F_{kk}}{\lambda_k} - \frac{1}{S}\sum_{i,k} \frac{A_{ii,kk}}{\lambda_k \lambda_i^2} + \frac{2}{S}\sum_{i,j,k}\frac{A_{ij,k}^2}{\lambda_i^2\lambda_j \lambda_k} -\frac{1}{S^2}\sum_k \frac{1}{\lambda_k}\left( \sum_i \frac{A_{ii,k}}{\lambda_i^2} \right)^2\\ 
&=
\sum_{i,k}\left(\frac{1}{\sigma_1} - \frac{1}{S\lambda_i^2}\right)\frac{A_{ii,kk}}{\lambda_k} +\sum_k\frac{F_{kk}}{\lambda_k}-\frac{1}{\sigma_1^2}\sum_k \frac{1}{\lambda_k}\left(\sum_{i}A_{ii,k}\right)^2 \\ 
&\qquad 
 -\frac{1}{S^2}\sum_k \frac{1}{\lambda_k}\left( \sum_i \frac{A_{ii,k}}{\lambda_i^2} \right)^2+ \frac{2}{S}\sum_{i,j,k}\frac{A_{ij,k}^2}{\lambda_i^2\lambda_j \lambda_k} + a\sum_k \frac{h_{kk}}{\lambda_k} + b\sum_k \frac{(|\nabla h|^2)_{kk}}{\lambda_k}.\\  
\end{split}
\end{equation}

Using $A_{ii,kk} = A_{kk,ii} + A_{ii} - A_{kk}$, we obtain, for each fixed $i$,
\begin{equation}\label{s4 contracted A}
\begin{aligned}
    \sum_k\frac{A_{ii,kk}}{\lambda_k}
           &=F_{ii}+\sum_{j,k}\frac{A_{jk,i}^2}{\lambda_j\lambda_k}
                     +\lambda_i S-n.
\end{aligned}
\end{equation}

Then \eqref{Q exact second derivative} becomes
\begin{equation}\label{Q after third derivatives}
    \begin{aligned}
0&\geq 
\sum_i W_i F_{ii} + \sum_{i,j,k}W_i \frac{A_{jk,i}^2}{\lambda_j \lambda_k}+ \sum_i W_i \lambda_i S - n\sum_i W_i \\ 
&\qquad 
-\sum_{i,j,k}\frac{A_{jk,i}^2}{\lambda_i \lambda_j \lambda_k}-\frac{1}{\sigma_1^2}\sum_k \frac{1}{\lambda_k}\left(\sum_{i}A_{ii,k}\right)^2 -\frac{1}{S^2}\sum_k \frac{1}{\lambda_k}\left( \sum_i \frac{A_{ii,k}}{\lambda_i^2} \right)^2 \\ 
&\qquad 
+ \frac{2}{S}\sum_{i,j,k}\frac{A_{ij,k}^2}{\lambda_i^2\lambda_j \lambda_k} + a(n-hS) + b\sum_k \frac{(|\nabla h|^2)_{kk}}{\lambda_k} \\ 
&= 
\sum_i W_i F_{ii} - \frac{n^2}{\sigma_1} + \frac{n}{S}\sum_i \frac{1}{\lambda_i^2} + a(n-hS) + b\sum_k \frac{(|\nabla h|^2)_{kk}}{\lambda_k} \\ 
&\qquad 
+\frac{1}{\sigma_1}\sum_{i,j,k}\frac{A_{ij,k}^2}{\lambda_i \lambda_j} - \frac{1}{\sigma_1^2}\sum_k \frac{1}{\lambda_k}\left(\sum_i A_{ii,k}\right)^2 \\ 
&\qquad 
+\frac{2}{S}\sum_{i,j,k}\frac{A_{ij,k}^2}{\lambda_i^2 \lambda_j \lambda_k} - \frac{1}{S}\sum_{i,j,k}\frac{A_{ij,k}^2}{\lambda_i \lambda_j \lambda_k^2} - \frac{1}{S^2}\sum_k \frac{1}{\lambda_k}\left(\sum_i \frac{A_{ii,k}}{\lambda_i^2}\right)^2 \\ 
&\geq 
\sum_i W_i F_{ii} - \frac{n^2}{\sigma_1} + \frac{n}{S}\sum_i \frac{1}{\lambda_i^2} + a(n-hS) + b\sum_k \frac{(|\nabla h|^2)_{kk}}{\lambda_k}.
    \end{aligned}
\end{equation}
The last inequality follows from Lemma \ref{prelim algebraic inequalities}, applied to {$\eta_{ijk}=A_{ij,k}$}.

Next, \eqref{prelim radial derivatives} gives
\[
    \sum_k\frac{P_{kk}}{\lambda_k}
       =2\sum_kh_kF_k+2\sigma_1-2nh,
\]
so we have
\begin{equation*}
    \begin{split}
    \sum_k\frac{(|\nabla h|^2)_{kk}}{\lambda_k}
       &= \sum_k \frac{P_{kk}}{\lambda_k} - \sum_k \frac{2h_k^2 + 2h(A_{kk}-h)}{\lambda_k}\\ 
       &= 
       2\sum_kh_kF_k+2\sigma_1-4nh+2h^2S
                       -2\sum_k\frac{h_k^2}{\lambda_k}.
    \end{split}
\end{equation*}
Inserting it into \eqref{Q after third derivatives}, we have
\begin{equation}\label{contracted Q}
\begin{aligned}
    0\geq{}&\sum_iW_iF_{ii}+2b\sum_kh_kF_k
             +2b\sigma_1-ahS+2bh^2S-2b\sum_i\frac{h_i^2}{\lambda_i}\\
           &+an-4bnh-\frac{n^2}{\sigma_1}
                    +\frac{n}{S}\sum_i\frac{1}{\lambda_i^2}.
\end{aligned}
\end{equation}

On the other hand, by \eqref{prelim radial derivatives}, we have
\begin{equation}\label{f second derivatives}
\begin{aligned}
    \widetilde{f}_i={}&(p-1)\frac{h_i}{h}
              +\beta\frac{\lambda_i h_i}{P}+(\log f)_i,\\
    \widetilde{f}_{ii}={}&\beta\left[\frac{\lambda_i^2}{P}
                         \left(1-\frac{2h_i^2}{P}\right)
                    +\frac{1}{P}\sum_kh_kA_{ii,k}
                    -\frac{h\lambda_i}{P}\right]\\
            &+(p-1)\left(\frac{\lambda_i}{h}-1-\frac{h_i^2}{h^2}\right)
                    +(\log f)_{ii}.
\end{aligned}
\end{equation}
Using \eqref{Q stationary} and \eqref{f second derivatives}, we have
\[
\begin{aligned}
    \frac{\beta}{P}\sum_kh_k\sum_iW_iA_{ii,k}
       &=-\frac{a\beta}{P}|\nabla h|^2
          -\frac{2b\beta}{P}\sum_k\lambda_kh_k^2
          +\frac{2b\beta h}{P}|\nabla h|^2,\\
    2b\sum_kh_k\widetilde{f}_k
       &=\frac{2b(p-1)}{h}|\nabla h|^2
          +\frac{2b\beta}{P}\sum_k\lambda_kh_k^2
          +2b\sum_kh_k(\log f)_k.
\end{aligned}
\]
Then
\eqref{contracted Q} gives
\begin{equation}\label{Q expanded right hand side}
\begin{aligned}
    0\geq{}&\frac{\beta}{P}\sum_iW_i\lambda_i^2
                                 \left(1-\frac{2h_i^2}{P}\right)
        +\left(\frac{p-1}{h}-\frac{\beta h}{P}\right)
                                 \sum_iW_i\lambda_i\\
       &-(p-1)\sum_iW_i\left(1+\frac{h_i^2}{h^2}\right)
        +\sum_iW_i(\log f)_{ii}\\
       &+\frac{\beta(2bh-a)}{P}|\nabla h|^2
        +\frac{2b(p-1)}{h}|\nabla h|^2
        +2b\sum_kh_k(\log f)_k\\
       &+2b\sigma_1-ahS+2bh^2S-2b\sum_i\frac{h_i^2}{\lambda_i}
        +an-4bnh-\frac{n^2}{\sigma_1}
                    +\frac{n}{S}\sum_i\frac{1}{\lambda_i^2}.
\end{aligned}
\end{equation}

To estimate the weighted terms, note that $S\lambda_i\geq1$ gives
$W_i\geq0$, and direct summation yields
\[
\begin{aligned}
    \sum_iW_i\lambda_i&=n,\\
\sum_iW_i&=\frac{n}{\sigma_1}+S-\frac{1}{S}\sum_i\lambda_i^{-2}
          \leq c_1+S\auditmathadd{,}\\
    \sum_iW_i\lambda_i^2
       &=\frac{\sum_i\lambda_i^2}{\sigma_1}+\sigma_1-\frac{n}{S}
          \leq2\sigma_1,
\end{aligned}
\]
where the positive constant {$c_1:=D_-^{-1/n}$}.
Since $h\geq m$, $P\geq m^2$ and $|1-2h_i^2/P|\leq1$, these imply
\[
\begin{aligned}
    \frac{\beta}{P}\sum_iW_i\lambda_i^2
                             \left(1-\frac{2h_i^2}{P}\right)
       &\geq-c_2\sigma_1,\\
    \left(\frac{p-1}{h}-\frac{\beta h}{P}\right)
                             \sum_iW_i\lambda_i
       &\geq-c_3,\\
    -(p-1)\sum_iW_i\left(1+\frac{h_i^2}{h^2}\right)
                        +\sum_iW_i(\log f)_{ii}
       &\geq-c_4(1+S),
\end{aligned}
\]
where the positive constants $c_2$ and $c_3$ depend only on $n,q,m$ and
$n,p,q,m,M$, respectively, while the positive constant $c_4$ depends only on
{$c_1,p,m,K,\min_{C_\theta}f$} and {$\Vert f\Vert_{C^2(C_\theta)}$}.
Here we use {$1+h_i^2/h^2\leq1+K^2/m^2$} and the {$C^2$} bound for {$\log f$}.
Since {$a<0$} and {$b>0$}, \eqref{Q expanded right hand side} yields
\begin{equation}\label{Q expanded right hand side 2}
    \begin{aligned}
        0&\geq -c_2\sigma_1 -c_3 -c_4(1+S) - c_5(|a|+b) \\ 
        &\qquad 
        +2b\sigma_1-ahS+2bh^2S-2b\sum_i\frac{h_i^2}{\lambda_i}
        +an-4bnh-\frac{n^2}{\sigma_1}
                    +\frac{n}{S}\sum_i\frac{1}{\lambda_i^2} \\ 
                    &\geq
                    -c_2\sigma_1 -c_3 -c_4(1+S) - c_5(|a|+b) - c_6(1+|a|+b) \\ 
                    &\qquad 
                    +2b\sigma_1 - ahS + 2bh^2 S - 2b\sum_i \frac{h_i^2}{\lambda_i} + \frac{n}{S}\sum_i \frac{1}{\lambda_i^2} \\ 
                    &\geq 
                    (2b - c_2)\sigma_1 - ahS + 2bh^2 S - 2b\sum_i \frac{h_i^2}{\lambda_i} - c_4 S - c_7(1+ |a| +b) \\ 
                    &\geq 
                    (2b-c_2)\sigma_1 - (am + 2bK^2 + c_4 )S - c_7(1+|a|+b),
    \end{aligned}
\end{equation}
where the positive constant $c_5$
depends only on $n,p,q,m,M,K,\min_{C_\theta}f$ and
$\Vert f\Vert_{C^1(C_\theta)}$, and we obtain the second inequality by
\[
    an-4bnh-\frac{n^2}{\sigma_1}\geq-c_6(1+|a|+b),
\]
where the positive constant $c_6$ depends only on $n,M,c_1$, and we set {$c_7: = c_3 + c_4 +c_5 + c_6$}.

Choose 
\[
    b:=\frac{c_2+1}{2}>0,\qquad
    a_1:=-\frac{2bK^2+c_4+1}{m}<0,
\]
so the positive constant $b$ depends only on $n,q,m$, and the negative
constant $a_1$ depends only on $n,p,q,m, M, K, \min_{C_\theta}f$ and
$\Vert f\Vert_{C^2(C_\theta)}$.
For every $a\leq a_1$, 
$$2b-c_2 \geq 1, \qquad -(am + 2bK^2 + c_4) \geq 1. $$
Hence, we have
\[
    \sigma_1(\xi_0)\leq c_7(1+|a|+b).
\]
We will fix $a$ after the boundary estimate.

\noindent\textbf{Case 2. $\xi_0\in\partial C_\theta$.}
By Proposition \ref{pro}, $A_{\alpha n}=h_{\alpha n}=0$ along the boundary,
where $e_n=\mu$ and $\alpha=1,\ldots,n-1$.
Choose the tangential frame so that $A$ is diagonal at $\xi_0$, and write
$\lambda_n=A_{nn}$. Since $\mu$ is the outward normal, at this maximum value point we have
\begin{equation}\label{Q boundary raw}
    \begin{aligned}
    0\leq Q_n & =\frac{(\sigma_1)_n}{\sigma_1}+F_n+\frac{S_n}{S}
                      +a\cot\theta\,h+b\left(|\nabla h|^2\right)_{n}\\ 
                      &= \frac{1}{\sigma_1}\sum_i A_{ii,n} + a\cot\theta h + 2b\cot\theta h(\lambda_n -h) + F_n + \frac{S_n}{S}. 
    \end{aligned}
\end{equation}

Taking a tangential derivative of $A_{\alpha n}=0$, using
$\nabla_{e_\alpha}\mu=\cot\theta\,e_\alpha$ and the Codazzi identity
\eqref{prelim Codazzi}, we find
\[
    A_{\alpha\alpha,n}=A_{\alpha n,\alpha}
       =\cot\theta\,(\lambda_n-\lambda_\alpha)
       \qquad(\alpha=1,\ldots,n-1).
\]
Since $F_n=\sum_i A_{ii,n}/\lambda_i$, it follows that
\[
\begin{aligned}
    A_{nn,n}&= \lambda_n F_n - \lambda_n \sum_{\alpha<n}\frac{A_{\alpha\alpha,n}}{\lambda_\alpha}\\ 
    &= 
    \lambda_nF_n
                   -\cot\theta\,\lambda_n(\lambda_nS-n).
\end{aligned}
\]
Thus
$$(\sigma_1)_n
       =\lambda_nF_n
           +\cot\theta\,(2n\lambda_n-\sigma_1-\lambda_n^2S),$$

      $$S_n=-\sum_i\frac{A_{ii,n}}{\lambda_i^2}
       =-\frac{F_n}{\lambda_n}
          -\cot\theta\,\lambda_n\sum_{\alpha<n}
                    \left(\frac{1}{\lambda_\alpha}
                                      -\frac{1}{\lambda_n}\right)^2.$$
Consequently, $\sigma_{n-1}(A)=(\det A)S$ gives
\[
\begin{aligned}
    (\log\sigma_{n-1}(A))_n
    ={}&\left(1-\frac{1}{\lambda_nS}\right)F_n\\
       &-\frac{\cot\theta\,\lambda_n}{S}
         \sum_{\alpha<n}
         \left(\frac{1}{\lambda_\alpha}
               -\frac{1}{\lambda_n}\right)^2.
\end{aligned}
\]
This identity uses only the Robin boundary condition and
the Codazzi identities, not the particular right-hand
side of Eq.\ \eqref{Eq}.

Hence, we have
\begin{equation}\label{boundary Q derivative}
\begin{aligned}
    0\leq\frac{\sigma_1}{\lambda_n}Q_n
    ={}&\left(1+\frac{\sigma_1}{\lambda_n}
                     -\frac{\sigma_1}{\lambda_n^2S}\right)F_n
         +\cot\theta\left(2n-\frac{\sigma_1}{\lambda_n}-\lambda_nS\right)\\
       &+a\cot\theta\,h\frac{\sigma_1}{\lambda_n}
         +2b\cot\theta\,h(\lambda_n-h)\frac{\sigma_1}{\lambda_n}\\
       &-\cot\theta\,\frac{\sigma_1}{S}
          \sum_{\alpha<n}\left(\frac{1}{\lambda_\alpha}
                                      -\frac{1}{\lambda_n}\right)^2 \\ 
                                      \leq&{} 
                                      \left(1+\frac{\sigma_1}{\lambda_n}
                     -\frac{\sigma_1}{\lambda_n^2S}\right)F_n+a\cot\theta\,h\frac{\sigma_1}{\lambda_n}
         +2b\cot\theta\,h(\lambda_n-h)\frac{\sigma_1}{\lambda_n}\\
       &-\cot\theta\,\frac{\sigma_1}{S}
          \sum_{\alpha<n}\left(\frac{1}{\lambda_\alpha}
                                      -\frac{1}{\lambda_n}\right)^2,
\end{aligned}
\end{equation}
where the last inequality follows from 
$$2n-\frac{\sigma_1}{\lambda_n}-\lambda_nS
      =-\sum_{\alpha<n}\left(\frac{\lambda_n}{\lambda_\alpha}
                            +\frac{\lambda_\alpha}{\lambda_n}-2\right)
        \leq 0.$$

Before estimating \eqref{boundary Q derivative}, note that
\begin{equation}\label{F boundary derivative}
    \begin{aligned}
    F_n&= \widetilde{f}_n = (p-1)\cot\theta
              +\beta\cot\theta\,\frac{h\lambda_n}{P}+(\log f)_n, \\ 
    1& \leq 1 + \frac{\sigma_1}{\lambda_n} - \frac{\sigma_1}{\lambda_n^2 S}\leq 1+\frac{\sigma_1}{\lambda_n},
    \end{aligned}
\end{equation}
where the second line uses $\lambda_n S \geq 1$, so we have
\begin{equation}\label{f_n's estimate}
|F_n| = |\widetilde{f}_n| \leq c_8(1+\lambda_n),
\end{equation}
where the positive constant $c_8$ depends only on
$n,p,q,\theta,m,M,\min_{C_\theta}f$ and $\Vert f\Vert_{C^1(C_\theta)}$.
We distinguish the cases {$\lambda_n\leq m/2$} and {$\lambda_n>m/2$}.

\noindent\emph{Subcase 2.1. $\lambda_n\leq m/2$.}
Here $\lambda_n-h\leq-m/2$, so 
$$2b\cot\theta h(\lambda_n -h )\frac{\sigma_1}{\lambda_n}<0.$$
In this subcase, by
\eqref{F boundary derivative}, {$|F_n|\leq c_9:=c_8(1+m/2)$}. Thus \eqref{boundary Q derivative} becomes
\begin{equation}\label{s4 small normal}
\begin{aligned}
    0&\leq c_9\left(1+\frac{\sigma_1}{\lambda_n}\right)
                 +a\cot\theta\,h\frac{\sigma_1}{\lambda_n}\\
     &\leq c_9+(c_9+am\cot\theta)\frac{\sigma_1}{\lambda_n}.
\end{aligned}
\end{equation}
Denote $a_2:=-(c_9+1)/(m\cot\theta)<0$, where the negative constant
$a_2$ depends only on
$n,p,q,\theta,m,M,\min_{C_\theta}f$ and $\Vert f\Vert_{C^1(C_\theta)}$.
For {$a\leq a_2$}, \eqref{s4 small normal} gives
\[
    \sigma_1(\xi_0)\leq c_9\lambda_n\leq c_9m/2.
\]

\noindent\emph{Subcase 2.2. $\lambda_n> m/2$.}
Since $\lambda_n\leq\sigma_1$, \eqref{F boundary derivative} gives
\[
    \left(1+\frac{\sigma_1}{\lambda_n}
                     -\frac{\sigma_1}{\lambda_n^2S}\right)F_n
       \leq c_8\left(1+\frac{\sigma_1}{\lambda_n}\right)(1+\lambda_n)
       \leq c_{10}(1+\sigma_1),
\]
where the positive constant {$c_{10}:=(2+2/m)c_8$}. The last step uses
\[
\displaystyle
    \left(1+\frac{\sigma_1}{\lambda_n}\right)(1+\lambda_n)
       =1+\lambda_n+\frac{\sigma_1}{\lambda_n}+\sigma_1
       \leq 1+(2+2/m)\sigma_1.
\]
Using this bound in \eqref{boundary Q derivative}, together with
$a<0$, we obtain
\begin{equation}\label{boundary inequality}
\begin{aligned}
    0\leq{}&c_{10}(1+\sigma_1)+2bM\cot\theta\,\sigma_1
          -\cot\theta\,\frac{\sigma_1}{S}
            \sum_{\alpha<n}\left(\frac{1}{\lambda_\alpha}
                                      -\frac{1}{\lambda_n}\right)^2\\
      \leq{}&c_{11}(1+b)(1+\sigma_1)
          -\cot\theta\,\frac{\sigma_1}{S}
            \sum_{\alpha<n}\left(\frac{1}{\lambda_\alpha}
                                      -\frac{1}{\lambda_n}\right)^2,
\end{aligned}
\end{equation}
where the positive constant {$c_{11}:=c_{10}+2M\cot\theta$}.
By the Cauchy--Schwarz inequality,
\[
    \sum_{\alpha<n}\left(\frac{1}{\lambda_\alpha}
                              -\frac{1}{\lambda_n}\right)^2
       \geq\frac{1}{n-1}\left(S-\frac{n}{\lambda_n}\right)^2.
\]
So \eqref{boundary inequality} yields
\[
    0\leq c_{11}(1+b)(1+\sigma_1)
             -\frac{\cot\theta}{n-1}\frac{\sigma_1}{S}
                         \left(S-\frac{n}{\lambda_n}\right)^2.
\]
If $S\geq4n/m$, then $\lambda_n > m/2$ implies
{$S-n/\lambda_n\geq S/2$}. Hence
\begin{equation}\label{s4 inverse trace bound}
    0\leq c_{11}(1+b)(1+\sigma_1)
               -\frac{\cot\theta}{4(n-1)}\sigma_1S.
\end{equation}
Since $\sigma_1\geq nD_-^{1/n}=n/c_1$, we have
$S\leq4(n-1)(\cot\theta)^{-1}\,c_{11}(1+b)(1+c_1/n)$. Together with the case $S<4n/m$, this yields
\[
    S\leq c_{12}:=\max\left\{\frac{4n}{m},\,
        4(n-1)(\cot\theta)^{-1}\,c_{11}(1+b)\left(1+\frac{c_1}{n}\right)\right\},
\]
where the positive constant $c_{12}$ depends only on
$n,m,\theta,c_1,c_{11},b$. 

Now $\lambda_i\geq S^{-1}$ and the determinant upper bound give
\[
    \lambda_i=\frac{\det A}{\prod_{j\neq i}\lambda_j}
        \leq D_+S^{n-1}\leq c_{13}:=D_+c_{12}^{n-1},
\]
hence {$\sigma_1(\xi_0)\leq nc_{13}$}. 

We now fix {$a=\min\{a_1,a_2\}<0$}. Combining the two cases gives
\[
    \sigma_1(\xi_0)\leq c_{14}:=
        \max\{c_7(1+|a|+b),\,c_9m/2,\,nc_{13}\},
\]
where the positive constant $c_{14}$ depends only on $n,p,q,\theta,m,M,K,\min_{C_\theta}f$ and
$\Vert f\Vert_{C^2(C_\theta)}$. At the maximum point, since $a<0$ and $\sigma_{n-1}(A)(\xi_0)\leq nc_{14}^{n-1}$, we have
$$Q(\xi_0) \leq \log n + n\log(1+c_{14})+ bK^2=:c_{15}.$$

On the other hand, the Maclaurin inequality gives
\[
    \sigma_{n-1}(A)\geq n(\det A)^{(n-1)/n}
        \geq nD_-^{(n-1)/n} =: c_{16}>0.
\]
For any $\xi\in C_\theta$, we have
\[
\begin{aligned}
     \log\sigma_1(\xi)&+\log c_{16}+aM \\ 
     &\leq Q(\xi)=\log\sigma_1(\xi)+ah(\xi) + b|\nabla h|^2(\xi) + \log \sigma_{n-1}(A(\xi))  \\ 
   &
    \leq Q(\xi_0)\leq c_{15}.
\end{aligned}
\]
Thus
\[
    \sigma_1\leq c_{16}^{-1}\exp(c_{15}-aM) = : c_{17}
        \qquad\text{on }C_\theta,
\]
where the positive constant $c_{17}$ depends only on $n,p,q,\theta,m,M,K,\min_{C_\theta}f$ and
{$\Vert f\Vert_{C^2(C_\theta)}$}. In particular,
\[
\displaystyle D_-c_{17}^{1-n}\leq\lambda_i\leq c_{17}
       \qquad(i=1,\ldots,n),
\]
since {$\lambda_i=\det A/\prod_{j\ne i}\lambda_j$}.
Thus the equation is uniformly elliptic. Since {$\nabla^2h=A-h\sigma$}, we obtain
$$\max_{C_{\theta}}|\nabla^2 h|\leq C,$$
where the positive constant $C$ depends only on
$n,p,q,\theta,m,M,K,\min_{C_\theta}f$ and
$\Vert f\Vert_{C^2(C_\theta)}$. 

This completes the proof.
\end{proof}

\begin{corollary}\label{c2 theorem}
Let $n\geq3$, $p>q$, and $\theta\in(0,\frac{\pi}{2})$. Suppose that $h$ is
a smooth positive strictly convex solution of Eq.\ \eqref{Eq}. Then, for
every $\alpha\in(0,1)$,
\[
    \min_{C_\theta}h\geq c,\qquad
    \Vert h\Vert_{C^{3,\alpha}(C_\theta)}\leq C,
\]
where the positive constants {$c$} and {$C$} depend only on
$n,p,q,\theta,\alpha$ and $f$.
\end{corollary}

\begin{proof}
Combining Lemma \ref{C0 estimate}, Lemma \ref{c1-1} and Theorem
\ref{c2 estimate}, we obtain
\begin{equation*}
\min_{C_\theta}h\geq c, \qquad \text{and}\qquad \Vert h\Vert_{C^{2}(C_\theta)}\leq C.
\end{equation*}
The preceding estimates also give
$C^{-1}\sigma\leq A\leq C\sigma$.
The logarithm of Eq.\ \eqref{Eq} is uniformly elliptic and concave
with respect to $\nabla^2h$ on the resulting compact admissible set,
and the Robin boundary operator is smooth and uniformly oblique.
The global second-derivative H\"older estimate
\cite[Theorem~6.2]{LT86} therefore gives
$\Vert h\Vert_{C^{2,\gamma}(C_\theta)}\leq C$ for some $\gamma\in(0,1)$.
Tangential differentiation and oblique Schauder estimates
\cite[Section~6.7]{GT83}, followed by recovery of the remaining
normal derivatives from the equation, yield higher-order bounds.
In particular, bootstrapping to $C^{4,\gamma}$ gives the stated
$C^{3,\alpha}$ bound for every $\alpha\in(0,1)$.

\end{proof}

\section{Proof of Theorem \ref{main1.1}} \label{se5}
We now prove Theorem \ref{main1.1} using the a priori estimates. We first establish existence by the continuity
method in logarithmic variables, and then prove uniqueness by the comparison principle.

For $h>0$, set
\[
    v=\log h,\qquad B[v]_{ij}=v_{ij}+v_iv_j+\sigma_{ij},
\]
and recall that $\beta=n+1-q$. Then
\[
    h_{ij}+h\sigma_{ij}=e^vB[v]_{ij},\qquad
    h^2+|\nabla h|^2=e^{2v}(1+|\nabla v|^2),
\]
while the boundary condition becomes $v_\mu=\cot\theta$.
All derivatives are taken with respect to {$\sigma$} in local orthonormal
frames, and repeated indices are summed from {$1$} to {$n$}.

Define
\[
    f^{(0)}=l^{1-p}(l^2+|\nabla l|^2)^{-\beta/2},\qquad
    f_t=(1-t)f^{(0)}+tf,\qquad 0\leq t\leq1.
\]
Replacing $f$ by $f_t$ in Eq.\ \eqref{Eq}, we obtain the equivalent problem
\begin{equation}\label{log path}
\left\{
\begin{aligned}
    \Phi_t[v]&:=\log\det B[v]
       -\frac{\beta}{2}\log(1+|\nabla v|^2)
       -(p-q)v-\log f_t=0 &&\text{in }C_\theta,\\
    v_\mu&=\cot\theta &&\text{on }\partial C_\theta.
\end{aligned}
\right.
\end{equation}
Here $B[v]>0$ is equivalent to the strict convexity of $h=e^v$.

\begin{proof}[Proof of Theorem~\ref{main1.1}]
\noindent\textbf{Existence.}
Fix $\alpha\in(0,1)$ and let
\[
    \mathcal V_\alpha
      =\{v\in C^{2,\alpha}(C_\theta):B[v]>0\}.
\]
This is an open subset of $C^{2,\alpha}(C_\theta)$. Define
\[
    \Psi(t,v)=(\Phi_t[v],v_\mu-\cot\theta),\qquad
    (t,v)\in[0,1]\times\mathcal V_\alpha,
\]
with values in
$C^{0,\alpha}(C_\theta)\times C^{1,\alpha}(\partial C_\theta)$, and set
\[
    \mathcal I=\{t\in[0,1]:\Psi(t,v)=0
                    \text{ for some }v\in\mathcal V_\alpha\}.
\]
Since $\nabla^2l+l\sigma=\sigma$ and $l_\mu=\cot\theta\,l$,
$h=l$ solves the equation with density $f^{(0)}$. Thus
$\log l\in\mathcal V_\alpha$ and $0\in\mathcal I$.

We next prove openness. At a solution $(t,v)$, differentiation in the
variation $\phi$ gives
\begin{equation}\label{linearized log}
    D_v\Phi_t[v]\phi
      =B^{ij}\phi_{ij}
       +\left(2B^{ij}v_i-\frac{\beta v_j}{1+|\nabla v|^2}\right)\phi_j
       -(p-q)\phi,
\end{equation}
where $(B^{ij})$ is the inverse of $(B[v]_{ij})$; the boundary
linearization is $\phi_\mu$.
The interior operator is uniformly elliptic, the boundary operator is
uniformly oblique, and the zeroth-order coefficient is $-(p-q)<0$.
The maximum principle and the Hopf boundary lemma
\cite[Lemma~3.4 and Theorem~3.5]{GT83} therefore give
$\ker D_v\Psi(t,v)=\{0\}$.
To obtain surjectivity, the linear oblique theory
\cite[Theorems~6.30 and 6.31]{GT83} first gives an isomorphism
\[
    \phi\longmapsto(D_v\Phi_t[v]\phi,\phi_\mu+\phi)
\]
from $C^{2,\alpha}(C_\theta)$ to
$C^{0,\alpha}(C_\theta)\times C^{1,\alpha}(\partial C_\theta)$.
The boundary trace is compact from $C^{2,\alpha}(C_\theta)$ to
$C^{1,\alpha}(\partial C_\theta)$. Hence $D_v\Psi(t,v)$ is a compact
perturbation of this isomorphism, and its trivial kernel implies
invertibility by the Fredholm alternative \cite[Theorem~5.3]{GT83}.
The implicit-function theorem now shows that $\mathcal I$ is relatively
open in $[0,1]$.

We turn to closedness. 
Every solution in $\mathcal V_\alpha$ is smooth up to the boundary.
Indeed, at a fixed solution, uniform ellipticity, tangential
difference quotients and oblique Schauder estimates
\cite[Section~6.7]{GT83}, followed by recovery of the remaining
normal derivatives from the equation, yield
$v\in C^\infty(C_\theta)$.

Since $f^{(0)}$ and $f$ are positive and smooth, for any $k\geq0$,
\[
\begin{aligned}
    \min_{C_\theta}f_t
       &\geq\min\{\min_{C_\theta}f^{(0)},\min_{C_\theta}f\}>0,\\
    \Vert f_t\Vert_{C^k(C_\theta)}
       &\leq\Vert f^{(0)}\Vert_{C^k(C_\theta)}
                    +\Vert f\Vert_{C^k(C_\theta)}.
\end{aligned}
\]
Thus the a priori estimates apply uniformly in {$t$}. Let $t_j\in\mathcal I$ tend to $t_\infty\in[0,1]$, choose
corresponding solutions $v_j$, and set $h_j=e^{v_j}$. By Lemma
\ref{C0 estimate} and Lemma \ref{c1-1}, Theorem \ref{c2 estimate} and
Corollary \ref{c2 theorem}, we have
\[
\begin{gathered}
    c_{18}^{-1}\leq h_j\leq c_{18},\qquad
    c_{18}^{-1}\sigma\leq\nabla^2h_j+h_j\sigma\leq c_{18}\sigma,\\
    \Vert h_j\Vert_{C^{3,\alpha}(C_\theta)}\leq c_{18},
\end{gathered}
\]
where the positive constant $c_{18}\geq1$ depends only on
$n,p,q,\theta,\alpha$ and $f$, and is independent of $j$.
Here the lower bound for $\nabla^2h_j+h_j\sigma$ follows from its upper
bound and the positive lower bound for its determinant.
By compactness, a subsequence satisfies
\[
    h_j\longrightarrow h_\infty
        \qquad\text{in }C^{2,\alpha}(C_\theta).
\]
The displayed bounds imply $h_\infty>0$ and
$\nabla^2h_\infty+h_\infty\sigma>0$. Passing to the limit in the equation
and boundary condition yields
$\Psi(t_\infty,\log h_\infty)=0$, so $t_\infty\in\mathcal I$.
Therefore $\mathcal I$ is closed, and hence $\mathcal I=[0,1]$.
Taking $t=1$ gives a positive smooth strictly convex solution of
Eq.\ \eqref{Eq}.

\smallskip
\noindent\textbf{Uniqueness.}
It remains to compare two solutions. Let
$v_0,v_1\in\mathcal V_\alpha$ solve \eqref{log path} for the same $t$,
and put $w=v_1-v_0$. Then $w_\mu=0$ on $\partial C_\theta$.
Denote {$B_0:=B[v_0]$} and {$B_1:=B[v_1]$}, and denote the inverse of $B_0$ by $(B_0^{ij})$.
By the concavity of $\log\det$,
\[
\begin{aligned}
    \log\det B_1-\log\det B_0
       &\leq B_0^{ij}(B_1-B_0)_{ij}\\
       &=B_0^{ij}w_{ij}+B_0^{ij}(v_{1,i}+v_{0,i})w_j,
\end{aligned}
\]
where the last equality uses the symmetry of $B_0^{ij}$. Also,
\[
    \log\frac{1+|\nabla v_1|^2}{1+|\nabla v_0|^2}
       =2\int_0^1
          \frac{v_{0,j}+s w_j}{1+|\nabla v_0+s\nabla w|^2}\,ds\,w_j.
\]
Subtracting the two equations and using these identities gives
\begin{equation}\label{comparison inequality}
\begin{aligned}
    0\leq{}&B_0^{ij}w_{ij}-(p-q)w\\
       &+\left(B_0^{ij}(v_{1,i}+v_{0,i})
          -\beta\int_0^1
              \frac{v_{0,j}+s w_j}{1+|\nabla v_0+s\nabla w|^2}\,ds
        \right)w_j.
\end{aligned}
\end{equation}
This is a uniformly elliptic linear inequality with bounded coefficients
and a negative zeroth-order coefficient. The maximum principle and the
Hopf boundary lemma, together with $w_\mu=0$, imply $w\leq0$.
Interchanging $v_0$ and $v_1$ gives $w\geq0$, and hence $v_0=v_1$.

Thus Eq.\ \eqref{Eq} has a unique positive smooth strictly convex
solution, completing the proof of Theorem \ref{main1.1}.
\end{proof}

%\section*{Acknowledgments}
% Acknowledgments and funding information will be supplied later.

\end{document}